\documentclass[11pt,A4paper]{article}
\usepackage{amsfonts}
\usepackage{amssymb}
\usepackage{amsthm}
\usepackage{amsmath}
\usepackage{graphicx}
\usepackage{graphics}
\usepackage{a4wide}
\usepackage{url}
\usepackage{comment}
\usepackage{tikz}
\usepackage{color}
\usepackage{xcolor}
\usepackage{soul}
\usepackage{natbib}
\setcitestyle{authoryear,open={(},close={)}}

\def\Exp{{\mathbb{E}}}
\newcommand{\cov}{{\rm cov}}
\newcommand{\corr}{{\rm corr}}

\newcommand{\R}{{\mathbb R}}

\newcommand{\cA}{{\cal A}}
\newcommand{\cY}{{\cal Y}}

\newcommand{\cX}{{\cal X}}

\newcommand{\cM}{{\cal M}}
\newcommand{\cF}{{\cal F}}
\newcommand{\cG}{{\cal G}}
\newcommand{\cC}{{\cal C}}

\newcommand{\cQ}{{\cal Q}}

\newcommand{\id}{{\bf I}}

\newcommand{\bx}{{\bf x}}
\newcommand{\bX}{{\bf X}}
\newcommand{\bY}{{\bf Y}}
\newcommand{\bZ}{{\bf Z}}

\newcommand{\bN}{{\bf N}}

\newcommand\independent{\protect\mathpalette{\protect\independenT}{\perp}}
\def\independenT#1#2{\mathrel{\rlap{$#1#2$}\mkern2mu{#1#2}}}

\definecolor{lightgray}{gray}{0.85}
\sethlcolor{Dandelion}

\usetikzlibrary{arrows,shapes,plotmarks,positioning}
\usetikzlibrary{decorations.markings}

\tikzset{>=stealth'} 
\tikzstyle{graphnode} = 
   [circle,draw=black,minimum size=22pt,text centered,text
     width=22pt,inner sep=0pt] 
\tikzstyle{var}   =[graphnode,fill=white]
\tikzstyle{vardashed}   =[graphnode,draw=gray,fill=white]
\tikzstyle{obs}   =[graphnode,fill=black,text=white]
\tikzstyle{obsgrey}   =[graphnode,draw=white,fill=lightgray,text=black]
\tikzstyle{par}    =[graphnode,draw=white,fill=red,text=black] 
 \tikzstyle{crucial} =[graphnode,draw=white,fill=yellow,text=black] 
\tikzstyle{fac}   =[rectangle,draw=black,fill=black!25,minimum size=5pt]
\tikzstyle{facprior} =[rectangle,draw=black,fill=black,text=white,minimum size=5pt]
\tikzstyle{edge}  =[draw=white,double=black,very thick,-]
\tikzstyle{blueedge}  =[draw=white,double=blue,very thick,-]
\tikzstyle{rededge}  =[draw=white,double=red,very thick,-]
\tikzstyle{prior} =[rectangle, draw=black, fill=black, minimum size=
5pt, inner sep=0pt]
\tikzstyle{dirprior} = [circle, draw=black, fill=black, minimum
size=5pt, inner sep=0pt]

\tikzstyle{dot_node}=[draw=black,fill=black,shape=circle]

\newtheorem{Def}{Definition}

\newtheorem{Thm}{Theorem}
\newtheorem{Lem}[Thm]{Lemma}

\newtheorem{Post}{Postulate}
\newtheorem{Ex}{Example}
\newtheorem{Option}{Option}

\date{May 17, 2018}

\begin{document}

\title{Merging joint distributions via\\ causal model classes with low VC dimension}

\author{Dominik Janzing\\
{\small dominik.janzing@tuebingen.mpg.de} \\
       {\small Max Planck Institute for Intelligent Systems}\\
       {\small Max-Planck-Ring 4}\\
       {\small 72076 T\"ubingen, Germany}}

\maketitle

\begin{abstract}
If $\bX,\bY,\bZ$ denote sets of random variables, two different data sources may
contain samples from $P_{\bX,\bY}$ and 
$P_{\bY,\bZ}$, respectively. 
We argue that causal inference can help inferring properties of the `unobserved joint distributions' $P_{\bX,\bY,\bZ}$ or
$P_{\bX,\bZ}$. 
The properties may be
conditional independences (as in `integrative causal inference') or also quantitative statements about dependences. 

More generally, we  define 
a learning scenario where the input is a subset of variables and the label is some statistical property of that subset. Sets of jointly observed variables
define the training points, while unobserved sets are possible test points. 
To solve this learning task,  we infer, as an intermediate step, a causal model from the observations 
that then entails properties of unobserved sets. 
Accordingly, we can define the VC dimension of
a class of causal models and derive generalization bounds for the predictions. 

Here, causal inference becomes more modest and better accessible to empirical tests than usual: rather than trying to find a causal hypothesis that is `true' (which is a problematic term when
it is unclear how to define interventions)
a causal hypothesis is {\it useful} whenever 
it correctly predicts statistical properties of unobserved joint distributions.

Within such a `pragmatic' application of causal inference, some popular  heuristic approaches 
become justified in retrospect. It is, for instance, allowed to infer DAGs from partial correlations instead of conditional independences if the DAGs are only used to predict partial correlations. 

I hypothesize that our pragmatic view on causality may even cover the usual meaning in terms of interventions and sketch why predicting the impact of interventions can sometimes also be phrased as a task of the above type.
\end{abstract}

\section{Introduction}

The difficulty of inferring causal relations from purely observational data lies in the fact that the observations drawn from a joint distribution $P_\bX$ with $\bX:=\{X_1,\dots,X_n\}$ are supposed to imply statements about how the system behaves under {\it interventions} \citep{Pearl2000,Spirtes1993}. More specifically,
one may be interested in the new joint distribution obtained by setting 
a subset $\tilde{\bX} \subset \bX$ of the variables to some specific values, which induces 
a different joint distribution.
If the task of causal inference is phrased this way, it actually lies outside the
typical domain of statistics. It thus requires   
assumptions that link statistics to causality
to render the task feasible under certain limitations. 
For instance, one can infer the causal directed acyclic graph (DAG) up to its Markov equivalence class from the observed conditional statistical independences
\citep{Spirtes1993,Pearl2000}. Moreover, on can also distinguish DAGs in the same Markov equivalence class when certain model assumptions such as linear models with non-Gaussian noise \citep{Kano2003} or non-linear additive noise \citep{Hoyer} are made. 

\paragraph{Relevance of causal information without reference to interventions}
The goal of causal inference need not necessarily consist in predicting the impact of interventions. 
Instead, causal information 
could help 
for transferring knowledge across data sets with different distributions  \citep{anticausal}. The underlying idea is
a modularity assumption  \citep[and references therein]{causality_book}
according to which 
only some conditional distributions in a causal Bayesian network may change  and others remain fixed.
Among many other tasks for which causal information could help, we should particularly emphasize so-called `integrative causal inference'  \citep{Tsamardinos}, which is
the work that is closest to the present paper.
 \citet{Tsamardinos} use causal inference to
combine knowledge from different data sets. The idea reads as follows:
Given some data sets $D_1,\dots,D_k$ 
containing observations from
different, but overlapping sets $S_1,\dots,S_k\subset \{X_1,\dots,X_n\}$ of variables. 
Then causal inference algorithms are applied independently to $S_1,\dots,S_k$. Afterwards,
a joint causal model is constructed that
entails
independences of 
some other subsets of variables of which 
no joint observations are available
(by slightly abusing terminology, we 
will refer to sets of variables that have not been observed together as `unobserved sets of variables', but keep in mind that although they
have not been observed {\it jointly},
they usually have been observed individually as part of some other observed set). 

To explain the idea more explicitly, we
 sketch Example~1 from \citet{Tsamardinos}, which combines knowledge from just two data sets\footnote{Note that combining information from data sources with overlapping variable sets has already been considered by \citet{Tillman2011,Tillman2014}, but the work of \citet{Tsamardinos} is closer to the present work due to its explicit goal of predicting unobserved statistical properties}. $D_1$ contains the variables
 $X,Y,W$ for which one observes 
$X \independent W\,|Y$ and no further (conditional or unconditional) independences. 
 The data set $D_2$ contains 
 the variables $X,W,Z$, where one observes
 $X\independent Q\,|Z$ as the only independence. 
 Then one constructs the set of all 
Maximal Ancestral Graphs (MAGs)\footnote{MAGs define a class of graphical causal models that is closed under marginalization and conditioning on
subsets of variables \cite{Richardson2002}.} 
 on the set $X,Y,Z,W$ 
that is consistent with the observed pattern of independences. 
As a result, the MAG implies 
 that $X \not\independent Y$, given any other subset of variables, although $X$ and $Y$ have never been observed together. 
 
From a higher-level perspective, the inference procedure thus reads:
 
\vspace{0.3cm} 
 
\begin{equation}\label{eq:scheme}
\begin{array}{c}\hbox{statistical properties of  observed subsets}    \\
\downarrow\\
\hbox{causal models consistent with those}\\
\downarrow\\
\hbox{statistical properties of unobserved subsets}\end{array}
\end{equation}

\vspace{0.3cm}

In contrast to \cite{Tsamardinos}, the term `statistical properties' need  not   necessarily
refer to conditional independences. On the one hand, there is meanwhile a broad variety of 
new approaches that infer causal directions 
from statistical properties other than conditional independences 
e.g., \cite{Kano2003,SunLauderdale,Hoyer,Zhang_UAI,
deterministic,SecondOrder,
JorisOliver,DiscrAN,Mooij2016,Marx2017}. On the other hand, the causal model inferred from the observations may entail statistical properties other than conditional independences -- subject to the model assumptions
on which the above-mentioned inference procedures rely. 

Regardless of what kind of statistical properties
are meant, the scheme in \eqref{eq:scheme}
describes a sense in which
a causal model that can be tested within the usual i.i.d.~scenario. This way, a causal model
entails statements that can be empirically tested
without referring to an interventional scenario. 
Consequently,
we drop the ambitious demand of finding `the true' causal model and replace it with the more modest goal of finding causal models that properly predict unseen joint distributions. 
After reinterpreting causal inference this way,
it also becomes directly accessible to 
statistical learning theory: 
assume we have a found a causal model that is consistent with the statistical properties of a large number of observed subsets,
we can hope that it also correctly predicts 
properties of unobserved subsets provided that
the causal model has been taken from
a sufficiently `small' class (to avoid overfitting).

This `radical empirical' point of view can be
developed even further: rather than asking 
whether some statistical property like
statistical independence is `true', we only ask whether the test at hand rejects or accepts it.\footnote{Asking whether two variables are 'in fact' statistically independent does not make sense for an empirical sample unless the sample
is thought to be part of an infinite sample
which is problematic in our finite world.}
Hence we can replace the term `statistical 
properties' in the scheme \eqref{eq:scheme}
with 'test results'. This point of view may also justify several common pragmatic 
solutions of the following issues:

\paragraph{Linear causal models for 
non-linear relations}  
Our perspective justifies to
apply multivariate Gaussian causal models to data sets that are clearly non-Gaussian:
Assume a hypothetical causal graph is inferred from 
the
conditional independence pattern obtained via {\it partial
correlation tests} (which is correct only for multivariate Gaussians), as done by common causal inference software
\cite{TETRAD}. Even if one knows that
the graph only represents
partial correlations correctly, but not conditional independences, it 
may predict well partial correlations
of unseen variable sets.
This way, the linear causal model can be helpful when the goal is only to predict linear statistics. This is good news 
particularly because general conditional independence tests remain a difficult issue, see, for instance, \citet{UAI_Kun_kernel}, for a recent proposal.

\paragraph{Tuning of confidence levels}
There is also another heuristic solution of a difficult question in causal inference that can be justified:
Inferring causal DAGs based on causal Markov condition and causal faithfulness \citep{Spirtes1993} relies on setting
the confidence levels for accepting conditional dependence. In practice, one will usually adjust the level such 
that enough independences are accepted and enough are rejected for the sample size at hand, otherwise inference is impossible. 
This is problematic, however,
from the perspective of the
common justification of causal faithfulness: if one rejects causal hypotheses with accidental conditional 
independences because they occur `with measure zero' \citep{Meek1995},
it becomes questionable to set the confidence level high enough just because one wants ot get some independences accepted.\footnote{For a detailed discussion of how causal conclusions 
of several causal inference algorithms
may repeatedly change after increasing the  sample size see \citep{Kelly2010}.}

Here we argue as follows instead: 
Assume we are given any arbitrary confidence level as threshold for the conditional independence tests. Further assume we have  found
a DAG $G$ from a sufficiently small model class
that is consistent with all 
the outcomes 'reject/accept' of
the conditional independence tests on a large number of subsets $S_1,\dots,S_k$. It is then justified to assume that $G$ will correctly predict the outcomes of this test
for unobserved variable sets $\tilde{S_1},\dots,\tilde{S}_l \subset S_1\cup \cdots \cup S_k$.

\paragraph{Methodological justification of
causal faithfulness}
In our learning scenarios, DAGs are used to predict for some 
choice of variables $X_{j_1},X_{j_2},\dots,X_{j_k}$
whether 
\[
X_{j_1} \independent X_{j_2}\,|X_{j_3}, \dots,X_{j_k}.
\]
Without faithfulness, the DAG can only entail {\it in}dependence, but never entail dependence. 
Rather than stating that 'unfaithful distributions are unlikely' we need faithfulness
simply to obtain a definite prediction in the first place.

\vspace{0.3cm}

The paper is structured as follows.
Section~\ref{sec:whycausal} explains why
causal models sometimes entail strong
statements regarding the composition of data sets. This motivates to use causal inference as an intermediate step when the actual task is to predict properties of unobserved joint distributions. 
Section~\ref{sec:formal} formalizes
our scenario as a standard prediction task where
the input is a subset (or an ordered tuple) of variables,
for which we want to test some statistical property. 
The output is a statistical property of that subset (or tuple). 
 This way, each observed variable set defines a {\it training} point for inferring the causal model while the unobserved variable sets are the {\it test} instances.
Accordingly, classes of causal models define
function classes, as described in Section~\ref{sec:VCdim}, whose richness can be 
measured via VC dimension. By straightforward
application of VC learning theory,
 Section~\ref{sec:generalization}
derives error  bounds for the predicted
statistical properties and discusses how they can be used as guidance for constructing causal hypotheses from not-too-rich classes of hypotheses. In Section~\ref{sec:interventions} we argue that our use of causal models is linked to the usual interpretation of causality in terms of interventions, which raises philosophical questions of whether the empirical content of causality reduces to providing rules on how to merge probability distributions.

\section{Why causal models are particularly helpful \label{sec:whycausal}}

It is not obvious why inferring properties of
unobserved joimnt distributions from observed ones should take the `detour' via causal models
visualized in \eqref{eq:scheme}.
One could also define a class of {\it statistical} models (that is, a class of joint distributions without any causal interpretation)
that is sufficiently small to yield definite predictions for the desired properties. 
The below example, however, suggests that
causal models typically entail particularly strong predictions regarding properties of the joint distribution. This is, among other reasons, because causal models
on subsets of variables sometimes imply 
a simple joint causal model. 
To make this point, consider the following toy example.

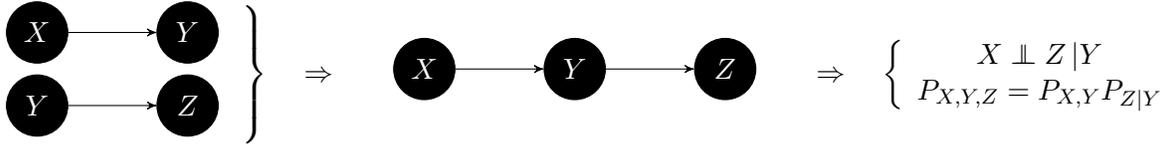
\begin{figure}
\[
\left. \begin{array}{c}
\hbox{
\begin{tikzpicture}
    \node[obs] at (0,0) (X) {$X$} ; 
    \node[obs] at (2,0) (Y) {$Y$} edge[<-] (X) ; 
\end{tikzpicture}
}
\\
\hbox{
\begin{tikzpicture}
    \node[obs] at (0,0) (Y) {$Y$} ; 
    \node[obs] at (2,0) (Z) {$Z$} edge[<-] (Y) ; 
\end{tikzpicture}
}
\end{array}
\right\} \quad \Rightarrow \quad 
\begin{array}{c}
\hbox{
\begin{tikzpicture}
    \node[obs] at (0,0) (X) {$X$} ; 
    \node[obs] at (2,0) (Y) {$Y$} edge[<-] (X) ;
    \node[obs] at (4,0) (Z) {$Z$} edge[<-] (Y) ; 
\end{tikzpicture}
}
\end{array}
\quad \Rightarrow \quad \left\{\begin{array}{c}
X \independent Z\,| Y \\ P_{X,Y,Z}=P_{X,Y} P_{Z|Y}
\end{array}\right.
\]
\caption{\label{fig:chain} Simplest example where causal information allows to `glue'
two distributions two a unique joint distribution.}
\end{figure}

\begin{Ex}[merging two cause-effect pairs to a chain]\label{ex:chain}
Assume we are given variables $X,Y,Z$ where we observed $P_{X,Y}$ and $P_{Y,Z}$.
The extension to $P_{X,Y,Z}$ is heavily underdetermined. Now assume that
we have the additional causal information that $X$ causes $Y$ and $Y$ causes $Z$ (see Figure~\ref{fig:chain}, left), in the sense that both pairs are causally sufficient. In other words, neither
$X$ and $Y$ nor $Y$ and $Z$ have a common cause.
This information can be the result of some bivariate causal inference algorithm that is able to exclude confunding. 
Given that there is, for instance, an additive noise model
from $Y$ to $Z$  \citep{Kano2003,Hoyer}, a confounder is unlikely
because it would typically destroy  
the independence of the additive noise term.

\vspace{0.2cm}
\noindent
{\bf Entire causal structure:}
We can then infer the entire causal structure to be the causal chain
$X\rightarrow Y\rightarrow Z$ for the following reasons. 
First we show that $X,Y,Z$ is a causally sufficient set of variables:
A common cause of $X$ and $Z$ would be a common cause of $Y$ and $Z$, too. The pair $(X,Y)$ and $(Y,Z)$ both have no common causes by assumption. 
One checks easily that no DAG with $3$ arrows leaves all $3$ pairs unconfounded.
Checking all DAGs on $X,Y,Z$ with $2$ arrows that have a path  from 
$X$ to $Y$ and from $Y$ to $Z$, we end up with the causal chain
in Figure~\ref{fig:chain}, middle, as the only option.

\vspace{0.2cm}
\noindent
{\bf Resulting joint distribution:}
This implies 
$
X \independent Z\,|Y.
$
Therefore,
$
P_{X,Y,Z}=P_{X,Y} P_{Z|Y}.
$
\end{Ex}
Note that our presentation of Example~\ref{ex:chain} neglected a subtle issue.
There are several different notions of
what it means that $X$ causes $Y$ in a {\it causally sufficient} way:
We have above used the purely graphical criterion asking whether there is some variable $Z$ having directed paths to $X$ and $Y$. An alternative option for defining that $X$ influences $Y$ in a causally sufficient way would be to demand that $P_Y^{ do(X=x)}=P_{Y|X=x}$.
This condition is called `interventional sufficiency' in 
\cite{causality_book}, a condition that
is testable by interventions on $X$ without referring to a larger background DAG in which $X$ and $Y$ are embedded.
This condition, however, is weaker than the graphical one and not sufficient for the above argument. This is because one could add the link $X \rightarrow Z$ to the chain $X\rightarrow Y \rightarrow Z$ 
and still observe that $P_Z^{do(Y=y)}=P_{Z|Y=y}$, as detailed by Example~9.2 in
\cite{causality_book}. 
Therefore, we stick to the graphical criterion of causal sufficiency and justify this by the fact that for `generic' parameter values it coincides with interventional sufficiency (which would actually be the more reasonable criterion).

\paragraph{Causal marginal problem vs. probabilistic marginal problem}
Given marginal distributions $P_{S_1},\dots,P_{S_k}$ on sets of variables,
the problem of existence and uniqueness 
of the joint distribution $P_{S_1\cup \cdots \cup S_k}$ (consistent with the marginals) is usually referred to as {\it marginal problem} \citep{Vorobev1962,Kellerer1964}. 
Here we will call it the {\it probabilistic}
marginal problem. 
Motivated by this terminology, we informally inroduce the
{\it causal marginal problem} as follows.
Given distributions $P_{S_1},\dots, P_{S_k}$
together with causal models $M_1,\dots,M_k$,
is there a unique joint distribution 
$P_{S_1\cup \cdots \cup S_k}$ with
causal model $M$ (consistent with the marginal model). 
The definition is informal because we have not specified our notion of `causal model'. 
Neither did we specify ``marginalization'' of causal models. For DAGs, marginalization requires the more general graphical model class MAGs  \citep{Richardson2002} already mentioned above,
while marginalization of structural equations
require structural equations with dependent noise terms \citep{Rubensteinetal17}.

Without formalizing this claim, Example~\ref{ex:chain} suggests that the causal marginal problem may have a unique solution even when the (probabilistic) marginal problem
doesn't \citep{causalMarginalTalk} -- subject to some genericity assumption explained above.  

The procedure for constructing the joint distribution in Example~\ref{ex:chain} can be described by the following special case of the scheme 
in \eqref{eq:scheme}:

\begin{equation}\label{eq:scheme2}
\begin{array}{c}\hbox{statistical properties of  observed subsets}    \\
\downarrow\\
\hbox{causal model for  observed subsets}    \\
\downarrow\\
\hbox{joint causal model}\\
\downarrow\\
\hbox{statistical properties of unobserved subsets}\end{array}
\end{equation}

Whether or not the joint causal model 
is inferred by first inferring `marginal' causal models for whether it is directly inferred 
from statistical properties of marginal
distributions will be irrelevant in our further discussion. In Example~\ref{ex:chain}, the detour 
over marginal causal models has been particularly simple.

\section{The formal setting \label{sec:formal}} 

Below we will usually refer to some given set of variables $S:=\{X_{j_1},\dots,X_{j_k}\}$ whose subsets are considered. Whenever this cannot cause any confusion, we will not carefully distinguish between the {\it set} $S$ and the {\it vector} $\bX:=(X_{j_1},\dots,X_{j_k})$ and also use the term
'joint distribution $P_S$' although the order of variables certainly matters.

\subsection{Statistical properties}
Statistical properties are the crucial concept of this work. On the one hand, they are used to infer causal structure. On the other hand, causal structure is used to predict them.
\begin{Def}[statistical property]
A 
statistical property $Q$ with range $\cY$ is
given by a function 
\[
Q: P_{Y_1,\dots,Y_k} \rightarrow \cY
\]
where $P_{Y_1,\dots,Y_k}$ denotes the joint distribution of $k$ variables
under consideration and $\cY$ some output 
space. Often we will consider
binary or real-valued properties, that is
$\cY=\{0,1\}$, $\cY=\{-1,+1\}$, or $\cY=\R$, respectively. 
\end{Def}  
By slightly abusing terminology, the term `statistical property' will sometimes refer
to the value in $\cY$ that is the output of
$Q$ or to the function $Q$ itself. This will, hopefully, cause no confusion. 

Here, $Q$ may be defined for fixed size
$k$ or for general $k$. Moreover, we will consider properties that depend on the ordering of the variables $Y_1,\dots,Y_k$, those that do not depend on it, or those that are invariant under some permutations $k$ variables. This will be clear from the context.
We will be refer to $k$ tuples for which
part of the order matters as `partly ordered 
tuples'.
To given an impression about the variety of 
statistical properties we conclude the section with a list of examples.

We start with an example for a binary property that
does not refer to an ordering:
\begin{Ex}[statistical independence]
\[
Q(P_{Y_1,\dots,Y_k}) = \left\{ \begin{array}{cc}
1 & \hbox{ for $Y_j$ jointly independent } \\
0 &  \hbox{ otherwise } 
\end{array}\right.
\]
\end{Ex}
The following binary property 
allows for some permutations of variables:
\begin{Ex}[conditional independence or partial uncorrelatedness]\label{ex:ci}
\[
Q(P_{Y_1,\dots,Y_k}) = \left\{ \begin{array}{ccc}
1 & \hbox{ for } &  Y_1 \independent Y_2 \,| Y_3,\dots,Y_k \\
0 &  \hbox{ otherwise } &
\end{array}\right.
\]
Likewise, $Q(P_{Y_1,\dots,Y_k})$ could indicate whether 
$Y_1$ and $Y_2$ have zero partial correlations, given $Y_3,\dots,Y_k$ (that is, whether they are uncorrelated after linear regression on $Y_3,\dots,Y_k$). 
\end{Ex}
To emphasize that our causal models are not only used to predict conditional independences but also other statistical properties we also mention linear additive noise models \citep{Kano2003}:
\begin{Ex}[existence of linear additive noise models]
\label{ex:lingam}
$Q(P_{Y_1,\dots,Y_k})=1$ if and only
if there is a matrix $A$ with entries $A_{ij}$,
that is lower triangular after permutation of basis vectors,
 such that 
\begin{equation}\label{eq:lingam}
Y_i = \sum_{j<i} A_{ij} Y_j +N_j,
\end{equation}
where $N_1,\dots,N_k$ are jointly independent noise variables. If no such additive linear model exists, we
set 
$Q(P_{Y_1,\dots,Y_k})=0$. 
\end{Ex}
Lower triangularity means that there is a DAG
such that $A$ has non-zero entries $A_{ij}$ whenever there is an arrow from $j$ to $i$.
Here, the entire order of variables matters.
Then \eqref{eq:lingam} is a linear structural equation.
Whenever the noise variables $N_j$ are non-Gaussian, linear additive noise models allow for 
the unique identification of 
the causal DAG 
\citep{Kano2003} if one assumes that the true generating process has been linear. Then, $Q(P_{Y_1,\dots,Y_k})=1$ holds for those orderings of variables that are compatible with the true DAG. This way, we have a statistical property that is directly linked to the causal structure (subject to a strong assumption, of course).

The following simple binary property will also
play a role later:
\begin{Ex}[sign of correlations]\label{ex:signcorrelations}
Whether a pair of random variables is positively or negatively correlated defines a simple binary
property in a scenario where all variables are correlated:
\[
Q(P_{Y_1,Y_2}) = \left\{\begin{array}{cl} 1 & \hbox{  if } \cov(Y_1,Y_2) >0\\
-1 & \hbox{ if } \cov(Y_1,Y_2)<0 
\end{array}\right.
\]
\end{Ex}
Finally, we mention a statistical property that is not binary but positive-semidefinite matrix-valued:
\begin{Ex}[covariances and correlations]\label{ex:cov}
For $k$ variables $Y_1,\dots,Y_k$
let $\cY$ be the set of positive semi-definite matrices. Then define
\[
Q: P_{Y_1,\dots,Y_k} \mapsto \Sigma_{Y_1,\dots,Y_k},
\]
where $\Sigma_{Y_1,\dots,Y_n}$ denotes the
joint covariance matrix of $Y_1,\dots,Y_n$.
For $k=2$, one can also get a real-valued property by focusing on the off-diagonal term.
One may then define a map $Q$  
\[
Q(P_{Y_1,Y_2}) := \cov(Y_1,Y_2),
\]
or alternatively, if one prefers correlations, define
\[
Q(P_{Y_1,Y_2}) := \corr(Y_1,Y_2).
\]
\end{Ex}

\subsection{Statistical and causal models}

The idea of this paper is that causal models are used to
predict statistical properties, but a priori, the models need not be causal. One can use Bayesian networks, for instance, to encode conditional statistical independences with or without interpreting the arrows as formalizing causal influence. For the formalism introduced in this section it does not matter whether one interprets the models as causal or not. 
Example~\ref{ex:chain}, however, suggested
that model classes that come with a causal semantics are particularly intuitive 
regarding the statistical properties they predict.
We now introduce our notion of `models':
\begin{Def}[models for a statistical property]
Given a set $S:=\{X_1,\dots,X_n\}$ of variables 
and some statistical property $Q$, 
 a model $M$ for $Q$ 
 is a class of joint distributions 
$P_{X_1,\dots,X_n}$ that coincide regarding the output of $Q$, that is,
\[
Q(P_{Y_1,\dots,Y_n}) = Q(P'_{Y_1,\dots,Y_n}) \quad \forall P_{Y_1,\dots,Y_n},P'_{Y_1,\dots,Y_n} \in M,
\]
where $Y_1,\dots,Y_k \in S$. 
Accordingly, the property 
$Q_M$  predicted by the model $M$ is given by a function
\[
(Y_1,\dots,Y_k) \mapsto Q_M\left[(Y_1,\dots,Y_k)\right] := Q(P_{Y_1,\dots,Y_n}),
\]
for all $P_{X_1,\dots,X_n}$ in $M$,
where $(Y_1,\dots,Y_k)$ runs over all
 allowed input (partly ordered) tuples of $Q$.
\end{Def}
Formally, the `partly ordered tuples' are
equivalence classes  in $S^k$, where equivalence 
corresponds to irrelevant reorderings of the tuple. To avoid cumbersome formalism, we will
just refer to them as `the allowed inputs'.

Later, such a model will be, for instance, a DAG $G$ and the property $Q$ formalizes all conditional independences that hold for the respective Markov equivalence class. 
To understand the above terminology, note that $Q$
receives a distribution as input and the output of $Q$ tells us the respective property
of the distribution (e.g. whether independence holds). In contrast, $Q_M$ receives a set of nodes (variables) of the DAG as inputs and 
tells us the property entailed by $M$.
The goal will be to find a model $M$ for which
$Q_M$ and $Q$ coincide for the majority of observed tuples of variables. 

Our most prominent example reads:
\begin{Ex}[DAG as model for conditional independences]\label{ex:dagsci}
Let $G$ be a DAG with nodes $S:=\{X_1,\dots,X_n\}$ and $Q$ be the set of conditional independences as in Example~\ref{ex:ci}.
Then, let $Q_G$ be the function on
$k$-tuples from $S$ defined by
\[
Q_G\left[(Y_1,\dots,Y_k)\right] :=0
\]
 if and only if the Markov condition implies $Y_1\independent Y_2\,|Y_3\dots,Y_k$, and
\[
Q_G\left[ (Y_1,\dots,Y_k)\right] :=1
\] 
otherwise. 
\end{Ex}
Note that $Q_G(.)=1$ does not mean that the
Markov condition implies dependence, it only
says that it does not imply independence.
However, if we think of $G$ as a causal DAG, the common assumption of causal faithfulness \citep{Spirtes1993} states that all dependences
that are allowed by the Markov condition occur in reality. Adopting this assumption, 
we will therefore interpret $Q_G$ as a function that predicts dependence or independence, instead of making no prediction otherwise.

We also mention a particularly simple
class of DAGs that will appear as an interesting example later: 
\begin{Ex}[DAGs consisting of a single colliderfree path]\label{ex:cfreepath}
Let $\cG$ be the set of DAGs that consist of a single colliderfree path 
\[
X_{\pi(1)} - X_{\pi(2)} - X_{\pi(3)} - \cdots
- X_{\pi(n)}, 
\]
where the directions of the arrows are such that
there is no variable with two arrowheads.
Colliderfree paths have the important property that any dependence between two non-adjacent nodes is screened off by any variable that lies between the two nodes, that is, 
\[
X_j \independent X_k|\, X_l,
\]
whenever $X_l$ lies between $X_j$ and $X_k$. 
If one assumes, 
in addition, that
the joint distribution is Gaussian,
the partial correlation
between $X_j$ and $X_k$, given $X_l$,
vanishes. This implies that 
the correlation coefficient of any two nodes is given by the product of pairwise correlations along the path:
\begin{equation}\label{eq:corrproduct}
\corr (X_j,X_k)
= \prod_{i=\pi^{-1}(j)}^{\pi^{-1}(k)-1} \corr (X_{\pi(i)},X_{\pi(i+1)}) =: \prod_{i=\pi^{-1}(j)}^{\pi^{-1}(k)-1} r_i. 
\end{equation}
 This follows easily by induction because $\corr(X,Z)=\corr (X,Y)\corr(Y,Z)$ for any three variables $X,Y,Z$ with $X\independent Z\,|Y$. 
Therefore, such a DAG, together with all the correlations between adjacent nodes, 
 predicts all pairwise correlations. 
We therefore specify our model by 
$M:=(\pi,r)$, that is, the ordering of nodes and correlations of adjacent nodes. 
\end{Ex}

The following example shows that a DAG can entail also properties that are more sophisticated than just conditional independences and correlations: 
\begin{Ex}[DAGs and linear non-Gaussian additive noise]
\label{ex:daglingam}
Let $G$ be a DAG with nodes $S:=\{X_1,\dots,X_n\}$ and $Q$ be the linear additive noise property in Example~\ref{ex:lingam}.
Let $Q_G$ be the function
 on
$k$-tuples from $S$ defined by
\[
Q_G((Y_1,\dots,Y_k)) :=1
\]
 if and only if the following two conditions hold:\\ 
 (1) $Y_1,\dots,Y_k$ is a
 causally sufficient subset from $S$ in $G$ and, that is, no two different $Y_i,Y_j$ have a common ancestor in $G$ \\
 (2) the ordering $Y_1,\dots,Y_k$ is consistent with $G$, that is,
 $Y_j$ is not ancestor of $Y_i$ in $G$
 for any $i<j$.    
\end{Ex}
Example~\ref{ex:lingam} predicts from the graphical structure whether the joint distribution of some subset of variables admits a linear additive noise model. The idea is the following. Assuming that the entire joint distribution of all $n$ variables has been generated by a linear additive noise model \citep{Kano2003}, any $k$-tuple 
$(Y_1,\dots,Y_k)$ also admits a linear additive noise model provided that (1) and (2) hold. 
This is because marginalizations of linear additive noise models remain linear additive noise models whenever one does not marginalize over common ancestors.\footnote{Note that the class of {\it non-linear} additive noise models \citep{Hoyer} is not closed under marginalization.}  
Hence, conditions (1) and (2) are clearly sufficient.
For generic parameter values of the underlying linear model the two conditions are also necessary because 
linear non-Gaussian models render
causal directions uniquely identifiable and also admit the detection of hidden common causes \citep{HoyerLatent08}.

\subsection{Testing properties on data}

So far we have introduced statistical properties as mathematical properties of distributions.
In real-world applications, however, we want to
predict the outcome of a test on empirical data.
The task is no longer to predict whether some set of variables is `really' conditionally independent, we just want to predict whether the statistical test at hand accepts independence.
Whether or not the test is appropriate for the respective mathematical property $Q$ is not relevant for the generalization bounds derived later. If one infers DAGs, for instance, by 
partial correlations and uses these DAGs only to infer partial correlations, it does not matter that non-linear relations actually prohibit to replace 
conditional independences with partial correlations. 
The reader may get confused by these remarks
because now there seems to be no requirement 
on the tests at all if it is not supposed to be a good test for the mathematical property $Q$.
This is a difficult question. One can say, however, that for a test that is entirely
unrelated to some property $Q$ we have no guidance what outcomes of our test a causal hypothesis should predict. The fact that partial correlations, despite all their limitations, approximate conditional independence, 
does provide some justification for expecting 
vanishing partial correlations in many cases where there is d-separation in the causal DAG.

We first specify the information provided by a data set.
\begin{Def}[data set]
 Each data set $D_j$ is an $l_j\times k_j$ matrix of observations, where $l_j$ denotes the sample size and $k_j$ the number of variables. 
 Further, the dataset contains a $k_j$-tuple
 of values from $\{1,\dots,n\}$ 
 specifying the $k_j$ variables $Y_1,\dots,Y_{k_j} \subset \{X_1,\dots,X_n\}$ the samples refer to. 
\end{Def}
To check whether the variables under consideration in fact satisfy the property predicted by the model
we need some statistical test (in the case of binary properties) or an estimator (in the case of real-valued or other properties). 
Let us say that we are given 
some test / estimator for a property $Q$, formally defined as follows:
\begin{Def}[statistical test / estimator for $Q$]
A test (respective estimator for non-binary properties) for the statistical property $Q$ 
with range $\cY$ is a map
\[
Q_T: D \mapsto Q_T(D) \in \cY, 
\]
where $D$ is a data set that involves 
the observed instances of
$Y_1,\dots,Y_n$, where $(Y_1,\dots,Y_k)$ is
a partly ordered tuple that defines
an allowed input of $Q$.
 $Q_T(D)$ is thought to indicate the outcome
of the test or the estimated value, respectively. 
\end{Def}

\subsection{Phrasing the task as standard prediction problem}

Our learning problem now reads: given the data sets
$D_1,\dots,D_l$ with the $k$-tuples $S_1,\dots,S_l$
of variables, find a model $M$ such that
$Q_M(S_j) = Q_T(D_j)$ for all data sets $j=1,\dots,l$ or, less demanding, for most of the data sets.
However, more importantly, we would like to choose $M$ such that $Q_M(S_j) =Q_T(D_{l+1})$ will also hold for
 a {\it future} data set $D_{l+1}$.  

The problem of constructing a causal model now becomes a standard learning problem
where the training as well as the test examples  
are {\it data sets}. Note that also 
\citet{Lopez2015} phrased a causal inference problem as standard learning problem. There,
the task was to classify two variables as `cause' and `effect' after getting
a large number of cause-effect pairs as training examples. Here, however, the data sets refer to observations from different subsets of variables that are assumed to follow a joint distribution
over the union of all variables occurring in any of the data sets.

Having phrased our problem as a standard prediction scenario whose inputs are subsets of variables, we now introduce the usual notion of empirical error on the training data accordingly:
\begin{Def}[empirical error]
Let $Q$ be a statistical property,
$Q_T$ a statistical test, and
$D:=\{D_1,\dots,D_k\}$ a collection of data sets referring to the variable tuples $S_1,\dots,S_k$. Then the empirical training error of model $M$ is defined by
\[
L(M) := \frac{1}{k} \sum_{j=1}^k |Q_T (D_j) - Q_M(S_{D_j})|. 
\]
\end{Def}
Finding a model $M$ for which the training error
is small does not guarantee, however, that 
the error will also be small for future test data. If $M$ has been chosen from a `too rich'
class of models, the small training error may be 
a result of overfitting. Fortunately we have phrased our learning problem in a way that
the richness of a class of causal models
can be quantified by standard concepts from statistical learning theory. 
This will be discussed in the following section.

\section{Capacity of classes of causal models \label{sec:VCdim}}

We have formally phrased our problem
as  a prediction problem where the task is to predict the outcome in $\cY$ of $Q_T$ for
some test $T$ applied to an unobserved variable set. 
We now assume that we are given a class of models $\cM$ defining 
statistical properties $(Q_M)_{M\in \cM}$ that are supposed to predict the outcomes of $Q_T$.

\subsection{Binary properties}
Given some binary statistical property, we can straightforwardly apply the notion of VC-dimension \cite{Vapnik}  to classes $\cM$ and define:
\begin{Def}[VC dimension of a model class for binary properties]
Let $S:=\{X_1,\dots,X_n\}$ a
set of variables and $Q$ be a binary property. Let
$\cM$ be a class of models for $Q$, that is, each $M\in \cM$ defines a map
\[
Q_M: (Y_1,\dots,Y_k) \mapsto Q_M\left[(Y_1,\dots,Y_k)\right] \in \{0,1\}. 
\]
Then the VC dimension of $\cM$ is the largest number $h$ such that there
are $h$ allowed inputs $S_1,\dots,S_h$
for $Q_M$ such that the
 restriction of all $M\in \cM$ to $S_1,\dots,S_h$
runs over all $2^h$ possible binary functions. 
\end{Def}
Since our model classes are thought to be given by causal hypotheses the following class  is our most important example although we will later further restrict the class to get stronger generalization bounds:
\begin{Lem}[VC dimension of 
conditional independences entailed by DAGs] 
Let $\cG$ be the set of DAGs with nodes
$X_1,\dots,X_n$. For every $G\in \cG$,
we define $Q_G$ as in Example~\ref{ex:dagsci}.
Then the VC dimension $h$ of
$(Q_G)_{G\in \cG}$ satisfies
\begin{equation}\label{eq:VCDAGs}
h\leq n \log_2 n + n(n-1)/2\in O(n^2).
\end{equation}
\end{Lem}
\begin{proof}
The number $N_n$ of DAGs on $n$ labeled nodes 
 can easily be upper bounded by the number of orderings times the number of choices to draw an edge or not.  This yields
$N_n < n! 2^{n(n-1)/2}$.
Using Stirling's formula we obtain
\[
n! < e^{1/(12n)} \sqrt{2\pi n} \left(\frac{n}{e}\right)^n < n^n, 
\] 
and thus $N_n< n^n 2^{n(n-1)/2}$.
Since the VC dimension of a class cannot be larger than the binary logarithm of the number of
elements it contains, \eqref{eq:VCDAGs} easily follows.
\end{proof}

\vspace{0.3cm}
Note that the number of possible conditional independence tests of the form $Y_1\independent Y_2\,|Y_3$ already grows faster than the VC dimension, namely with the third power.
Therefore, the class of DAGs indeed defines a restriction since it is not able to explain 
all possible patterns of conditional (in)dependences even 
when one conditions on one variable only.

Nevertheless, the set of all DAGs may be too large for the number of data sets at hand. We therefore mention the following more restrictive class given by so-called polytrees, that is, DAGs whose skeleton is a tree (hence they contain no undirected cycles). 
\begin{Lem}[VC dimension of 
cond.~independences entailed by
polytrees]\label{lem:polytrees}
Let $\cG$ be the set of polyntrees with nodes
$X_1,\dots,X_n$. For every $G\in \cG$,
we define $Q_G$ as in Example~\ref{ex:dagsci}.
Then the VC dimension $h$ of
$(Q_G)_{G\in \cG}$ satisfies
\begin{equation}\label{eq:polytrees}
h\leq n (\log_2 n +1). 
\end{equation}
\end{Lem}
\begin{proof}
 According to Cayley's formula, the number of trees with $n$ nodes reads
$n^{n-2}$ \citep{Aigner1998}. The number of Markov equivalence classes of polytrees can be bounded from above by 
$
2^{n-1} -n +1
$
 \citep{Radhakrishnan2017}.
Thus the number of Markov equivalence classes
of polytrees is upper bounded by
\begin{equation}\label{eq:numberOffunc}
n^{n-2} (2^{n-1} -n +1)\leq n^{n-2} 2^n.
\end{equation}
 Again, the bound follows by taking the logarithm.
\end{proof}

\vspace{0.3cm}

We will later use the following result:
\begin{Lem}[VC dimension of sign of correlations along a path]\label{Lem:VCcorrPath}
Consider the set of DAGs on $X_1,\dots,X_n$ that consist of a single
colliderfree path as in Example~\ref{ex:cfreepath} and assume multivariate Gaussianity. The sign of pairwise 
correlations is then determined by
the permutation $\pi$ that aligns the graph
and the sign of correlations of all adjacent pairs. We thus parameterize a model by $M:=(\pi,s)$ where the vector $s:=(s_1,\dots,s_n)$ denotes the signs of adjacent nodes. 
The full model class $\cM$ is obtained when $\pi$ runs over the entire group of permutations and $s$ over all combinations in $\{-,1+1\}^n$. 
Let $Q$ be the
property indicating the sign of the correlation of any two variables as in
Example~\ref{ex:signcorrelations}. 
Then the VC dimension of $(Q_M)_{M\in \cM}$ is at most $n$.
\end{Lem}
\begin{proof}
Defining 
\[
s_j := \prod_{i=1}^{\pi^{-1}(j)-1} {\rm sign}( \corr (X_{\pi(i)},X_{\pi(i+1)}))
\]
we obtain 
\[
{\rm sign}( \corr(X_i,X_j) ) = s_i s_j,
\]
due to \eqref{eq:corrproduct}.
Therefore, the signs of all can be computed from
$s_1,\dots,s_n$. Since there are $2^n$ possible assignments for these values, $\cG$ thus induces $2^n$ functions and thus the VC dimension is at most $n$.
\end{proof}

\subsection{Real-valued statistical properties}

We also want to obtain quantitative statements about the strength of dependences and therefore
consider also the correlation as an example of a real-valued property.
\begin{Lem}[correlations along a path]
Let $\cM$ be the model class whose elements
$M$ are colliderfree paths together with a list of all correlations of adjacent pairs of nodes, see
Example~\ref{ex:cfreepath}. Assuming also multi-variate Gaussianity, $M$, again, defines
all pairwise correlations and we 
can thus define the model induced property
\[
Q_M\left[(X_j,X_k)\right]:= \corr_M (X_j,X_k),
\]
where the term on the right hand side 
denotes the correlation determined  by the model $M:=(\pi,r)$ as introduced in Example~\ref{ex:cfreepath}.  
Then the VC dimension of $(Q_M)_{M \in \cM}$ is in $O(n)$.
\end{Lem}
\begin{proof} We assume, for simplicity, that
all correlations are non-zero. 
To specify the absolute value of the correlation between
adjacent nodes we define the parameters
\[
\beta_{i} := \log |\corr_M  (X_{\pi(i-1)},X_{\pi(i)})|.
\]
To specify the sign of those correlations we define the binary values
\[
{\rm g}_{i} :=     \left\{\begin{array}{cc}
1 & \hbox{ for }  \corr_M (X_{\pi(i-1)},X_{\pi(i)})  < 0  \\
0 & \hbox{ otherwise } \end{array}\right.,
\]
for all $i\geq 2$. 

It will be convenient to introduce the parameters
\[
\alpha_j := \sum_{i =2}^{j} \beta_i,
\]
which are cumulative versions of the `adjacent log correlations' $\beta_i$.
Likewise, we introduce the binaries
\[
s_j := \left(\sum_{i =2}^{j} g_i\right) {\rm mod }\, 2,
\]
which indicate whether the number of negative correlations along the chain from its beginning 
is odd or even. 

This way, the correlations between any two nodes can be computed from $\alpha$ and $s$:
\[
\corr_M (X_j,X_k) = (-1)^{s_{\pi^{-1}(j)}+s_{\pi^{-1}(k)}} \, e^{|\alpha_{\pi^{-1}(j)} -\alpha_{\pi^{-1}(k)}|}. 
\]
For technical reasons we define $\corr$ formally as a 
function of {\it ordered} pairs of variables
although it is actually symmetric in $j$ and $k$. 
We are interested in the VC dimension of the family  $F:=(f_M)_{M\in \cM}$ 
of real-valued functions defined by
\[
f_{M}(j,k):= \corr_{M} (X_j,X_k)=:\rho^M_{i,j}.
\]
Its VC-dimension is defined as the VC dimension of the
set of classifiers $C:=(c^\gamma_{M})_{M,\gamma}$ with
\[
c^\gamma_{M} (j,k) :=   \left\{\begin{array}{cc}
1 & \hbox{ for }  \rho^M_{j,k}  \geq  \gamma  \\
0 & \hbox{ otherwise } \end{array}\right.,
\]
To estimate the VC dimension of $C$ we 
compose it from classifiers whose VC dimension is easier to estimate. 

We first define the family of classifiers given by
$C^>:=(c^{>\theta}_\alpha)_{\alpha\in \R^-,\theta\in \R}$ with
\[
c^{>\theta}_\alpha (j,k):=   \left\{\begin{array}{cc}
1 & \hbox{ for }  \alpha_{\pi^{-1}(j)}-\alpha_{\pi^{-1}(k)} \geq  \theta  \\
0 & \hbox{ otherwise } \end{array}\right..
\] 
Likewise, we define
$C^<:=(c^{<\theta}_\alpha)_{\alpha\in \R^-,\theta\in \R}$ with
\[
c^{<\theta}_\alpha (j,k):=   \left\{\begin{array}{cc}
1 & \hbox{ for }  \alpha_{\pi^{-1}(j)}-\alpha_{\pi^{-1}(k)} <  \theta  \\
0 & \hbox{ otherwise } \end{array}\right..
\] 
The VC dimensions pf $C^>$ and $C^<$ are at most $n+1$ because they are given by linear functions on the space of all possible $\alpha \in \R^n$ \citep{Vapnik1995}, Section 3.6, Example 1.  
Further, we define a set of classifiers that classify only according to the sign of the correlations:
\[
S:= (c^M_+) \cup (c^M_-), 
\]
where 
\[
c^M_+(j,k) := \left\{ \begin{array}{cc} 1 & \hbox{ if } \rho^M_{j,k} \geq 0 \\
0 & \hbox{ otherwise } \end{array}\right..
\]
Likewise, we set
\[
c^M_-(j,k) := \left\{ \begin{array}{cc} 1 & \hbox{ if } \rho^M_{j,k} < 0 \\
0 & \hbox{ otherwise } \end{array}\right..
\]
Since both components of $S$ have VC dimension $n$ at
most, the VC dimension of $S$ is in $O(n)$.

For $\gamma> 0$, $\rho^M_{j,k} \geq \gamma$ is equivalent to
\[
(\rho^M_{j,k} \geq 0) \wedge (\alpha_{\pi^{-1}(j)}-\alpha_{\pi^{-1}(k)} \geq \log \gamma)  \wedge
 (\alpha_{\pi^{-1}(k)} -\alpha_{\pi^{-1}(j)} \geq \log \gamma).   
\]
 Therefore, 
\[
c^\gamma_M \in S  \sqcap C^> \sqcap C^<,
\] 
for all $\gamma >0$, where  $\sqcap$ denotes the intersection of
`concept classes' \citep{vanDerVaart2009}  given by 
\[
C_1 \sqcap C_2:= (c_1 \cap c_2)_{c_1\in C_1,c_2 \in C_2}.
\]
Likewise, the union of concept classes is given by
\[
C_1 \sqcup C_2:= (c_1 \cup c_2)_{c_1\in C_1,c_2 \in C_2},
\]
as opposed to the set-theoretic unions and intersections.

For $\gamma <0$, $\rho^M_{j,k}\geq \gamma$ is equivalent to
\[
(\rho^M_{j,k} \geq 0) \vee \left\{ (a_{\pi^{-1}(j)} - \alpha_{\pi^{-1}(k)} \geq \log |\gamma |) \wedge (\alpha_{\pi^{-1}(k)} -\alpha_{\pi^{-1}(j)} \geq \log |\gamma|)\right\}.
\]
Hence, 
\[
c^\gamma_M \in S  \sqcup [ C^> \sqcap C^<],
\]
for all $\gamma<0$.
We then obtain:
\[
C \subset (S \sqcap C^> \sqcap C^<) \cup (S \sqcup [ C^> \sqcap C^<]).
\]
Hence, $C$ is a finite union and intersection of concept classes and set theoertic union,
each having VC dimension in $O(n)$. 
Therefore, $C$ has VC dimension in $O(n)$ \citep{vanDerVaart2009}.
 \end{proof}

\section{Generalization bounds \label{sec:generalization}}

\subsection{Binary properties\label{subsec:VCboundsbinary}}

After we have seen that in our scenario causal models
like DAGs define classifiers
in the sense of standard learning scenarios,
we can use the usual VC bounds like  Theorem~6.7 in \cite{Vapnik06} to guarantee generalization to future data sets. To this end, we need to assume that the data sets are sampled from some distribution of data sets, an assumption that will be discussed at the end of this section. 
\begin{Thm}[VC generalization bound]
\label{thm:VCbound}
Let $Q_T$ be a statistical test for some statistical binary property and
$\cM$ 
be a model class with VC dimension $h$
defining some model-induced property $Q_M$.  
Given $k$ data sets $D_1,\dots,D_k$ sampled from distribution $P_D$. Then
\begin{equation}\label{eq:VCbinarygeneral}
\Exp\left[|Q_T(D)-Q_M(D)|\right] \leq    \frac{1}{k} \sum_{j=1}^k |Q_T(D_j)-Q_M(S_{D_j})|
+ 2\sqrt{\frac{h\left(\ln \frac{2k}{h} +1\right) - \ln \frac{\eta}{9}}{k}}
\end{equation}
with probability $1-\eta$. 
\end{Thm}  
It thus suffices to increase the number of data sets slightly faster than the VC dimension. 

To illustrate how to apply Theorem~\ref{thm:VCbound} we recall the class of
polytrees in Lemma~\ref{lem:polytrees}.
An interesting property of polytrees
is that every pair of non-adjacent nodes can already be rendered conditional independent by one appropriate intermediate node. This is because there is always at most one (undirected) path
connecting them. Moreover, 
for any two nodes $X,Y$ that are not too close together
in the DAG, there 
is a realistic chance
that some randomly chosen $Z$ satisfies
$X\independent Y\,|Z$. 
Therefore, we consider the following scenario:

\begin{enumerate}
\item
Draw $k$ triples
$(Y_1,Y_2,Y_3)$ uniformly at random and check whether
$
Y_1 \independent Y_2\,| Y_3.
$
\item Search for a polytree $G$
that is consistent with the $k$ observed
(in)dependences.
\item Predict conditional independences
for unobserved triples via $G$
\end{enumerate}


Since the number of points in the training set should increase slightly faster than
the VC dimension (which is $O(n)$, see Lemma~\ref{lem:polytrees}), 
we know that a small fraction of
the possible independence tests (which grows with third power) is already sufficient to predict conditional further independences.

The red curve in Figure~\ref{fig:numberOfTests} provides a rough estimate of how $k$ needs to grow if we want to ensure that the term $\sqrt{.}$ in
\eqref{eq:VCbinarygeneral} is below $0.1$ for $\eta =0.1$. 
The blue curve shows how the number of possible tests grows, which significantly exceeds the required ones after $n=40$.
\begin{figure}
\centerline{
\includegraphics[width=0.8\textwidth]{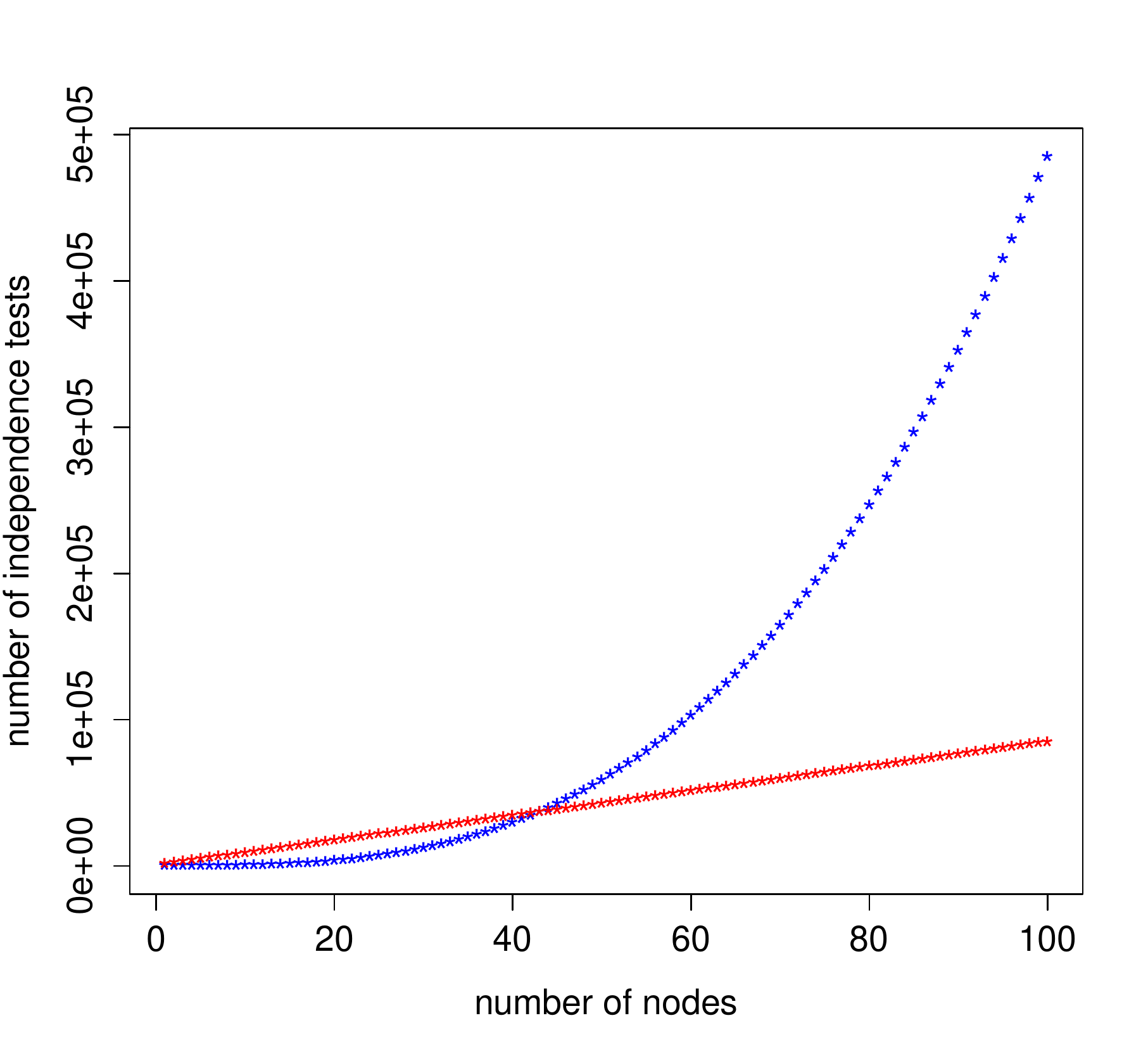}
}
\caption{\label{fig:numberOfTests}  The red curce shows how the number of tests required by the VC bound grows with the number of variables, while the blue one shows how the number of possible tests grows.}
\end{figure}
 For more than $100$ variables, only a fraction of about $1/4$ of the possible tests is needed to predict that also the remaining ones will hold with high probability.

While conditional independences have been used for causal inference already since decades, more recently it became popular to use other properties of distributions to infer causal DAGs. In particular, several methods have been proposed that distinguish between cause and effect from bivariate distributions, e.g.,
\cite{Kano2003,Hoyer,Zhang_UAI,deterministic,discreteAN,Lopez2015,Mooij2016}. 
It is tempting to do multivariate causal inference by 
finding DAGs that are consistent with the bivariate causal direction test.
This motivates the following example.
\begin{Lem}[bivariate directionality test on DAGs]
Let $\cG$ be the class of DAGs on $n$ nodes for which there is 
a directed path between all pairs of nodes.
Define a model-induced property $Q_G$ by
\[
Q_G(X_i,X_j):=\left\{\begin{array}{cc} 1 & \hbox{ iff there is a directed path from $X_i$ to $X_j$}\\
 -1 & \hbox{ iff there is a directed path from $X_j$ to $X_i$
 }\end{array}\right. 
\]
The VC-dimension of $(Q_G)_{G\in \cG}$ is 
at most $n-1$. 
\end{Lem}
\begin{proof}
The VC dimension is the maximal number $h$ of 
pairs of variables for which the causal directions can be oriented in all $2^h$ possible ways. 
If we take $n$  or more pairs, the undirected graph 
 defined by connecting each pair contains a cycle 
\[
(X_1,X_2),(X_2,X_3),\dots,
(X_{l-1},X_l),(X_l,X_1),
\]
with $l\leq n$.
Then, however, not all $2^l$ causal directions are possible because 
\[
X_1 \to X_2 \to \cdots X_l \to X_1
\]
would be a directed cycle.
\end{proof}
This result can be used to infer causal directions for pairs that have not been observed together:
\begin{enumerate}

\item Apply the bivariate causality test $Q_T$ to $k$ randomly chosen ordered pairs, where $k$ needs to grow slightly faster than $n$.

\item Search for a DAG $G\in \cG$ that is consistent with a last fraction of the outcomes.

\item Infer the outcome of further bivariate causality tests from $G$. 

\end{enumerate}

It is remarkable that the generalization bound holds regardless of how bivariate causality is tested and whether one understands which statistical features are used to infer the causal direction. Solely the fact that 
a causal hypothesis from a class of low VC dimension matches the majority of the bivariate tests ensures that it generalizes well to future tests. 

\subsection{Real-valued properties}

The VC bounds in Subsection~\ref{subsec:VCboundsbinary} referred to binary statistical properties. To consider
also real-valuedd properties note
that 
the VC dimension of  class of real-valued functions $(f_\lambda)_{\lambda\in \Lambda}$ with $f:\cX \rightarrow \R$
is defined as the VC dimension of
the set of binary functions, see Section~3.6 \cite{Vapnik1995}:
\[
\big(f_\lambda^{-1}\left((-\infty, r]\right)\big)_{\lambda \in \Lambda,r \in \R}.
\]
By combining (3.15) with (3.14) and (3.23) in \cite{Vapnik1995} we obtain:
\begin{Thm}[VC bound for real-valued statistical properties] Let $(Q_M)_{M\in \cM}$
be a class of $[A,B]$-valued model-induced 
properties with VC dimension $h$. 
Given $k$ data sets $D_1,\dots,D_k$ sampled from some distribution $P_D$. Then
\[
\Exp[|Q_T(D)-Q_M(D)|] \leq    \frac{1}{k} \sum_{j=1}^k |Q_T(D_j)-Q_M(S_{D_j})|
+ (B-A) \sqrt{\frac{h\left(\ln \frac{k}{h}+1\right) -\ln \frac{\eta}{4}}{k}}
\]
with probability at least $1-\eta$.
\end{Thm}
This bound can easily be applied to the prediction of correlations via collider-free paths: Due to Lemma~\ref{Lem:VCcorrPath}, we then have $h\in O(n)$. Since correlations are in $[-1,1]$, we can set $2$  for $B-A$.

\paragraph{Interpretation of the i.i.d. setting in learning theory}
In practical applications, the scenario is
usually somehow different because
one does not choose `observed' and `unobserved' subsets randomly. Instead, the observed sets are defined by the available data sets. One may object that
the above considerations are therefore inapplicable. There is no formal argument against this objection. However, there may be reasons to believe that the observed variable sets at hand are not substantially different from the unobserved ones whose properties are supposed to be predicted,
apart from the fact that they are observed. Based on this belief, one may still use the above generalization bounds 
as guidance on the richness of the class of causal hypotheses that is allowed to obtain good generalization properties.

\section{Predicting impact of interventions by merging distributions \label{sec:interventions}}

We have argued that causal hypotheses
provide strong guidance on how to merge probability distributions and thus
become empirically testable without resorting to interventions. One may wonder whether this view on causality is completely disconnected to 
interventions. Here I argue that it is not. In some sense, estimating the impact of an intervention can also be phrased as the problem of inferring properties of unobserved joint distributions.

Assume we want to test whether the causal hypothesis $X\to Y$ is true. We would then check how the distribution of $Y$ changes under randomized interventions on $X$. Let us formally introduce a variable $F_X$ \citep{Pearl2000} that 
can attain all possible values $x$ of $X$ 
(indicating to which value $x$ is set to)
or the value ${\tt idle}$ (if no intervention is made).  Whether $X$ influences $Y$ is then equivalent to 
\begin{equation}\label{eq:notind}
F_X \not\independent Y.
\end{equation}
If we demand that this causal relation is unconfounded (as is usually intended by the notation $X\to Y$), we have to test the condition
\begin{equation}\label{eq:notconf}
P_{Y|F_X=x} = P_{Y|X=x}.
\end{equation}
Before the intervention is made,
both conditions \eqref{eq:notind} and \eqref{eq:notconf} refer to the unobserved distribution $P_{Y,F_X}$. 
Inferring whether $X\to Y$ is true from $P_{X,Y}$ thus amounts to inferring 
the unobserved distribution $P_{Y,F_X}$
from $P_{X,Y}$ plus the additional background knowledge regarding the statistical and causal relation between $F_X$ and $X$ (which is just based on the knowledge that the action we made has been in fact the desired intervention).
In applications it can be a non-trivial question why some action can be considered an intervention on a target variable at hand (for instance in complex gene-gene interactions). 
If one assumes that it is based on purely observational data (maybe earlier in the past),
we have reduced the problem of predicting the impact of interventions entirely to the problem of merging joint distributions.

\section{Conclusions}

We have described different scenarios where causal models can be used to infer statistical properties
of joint distributions of variables that have never been observed together. If 
the causal models are taken from a class of
sufficiently low VC dimension, this can be justified by generalization bounds from statistical learning theory. 

This opens a new pragmatic and context-dependent perspective on causality
where the essential empirical content of a causal model may consist in its prediction
regarding how to merge distributions from overlapping data sets. Such a pragmatic use of
causal concepts may be helpful for domains where 
the interventional definition of causality
raises difficult questions
(if one claims that the age of a person causally influences his/her income, as assumed in \citet{Mooij2016}, it is unclear  
what it means to intervene on the variable 'Age'). We have, moreover, argued that 
our pragmatic view of causal models is related to the usual concept of causality in terms of
interventions. 

It is even possible that this view on causality could also be relevant for foundational questions of physics, where 
the language of causal models
plays an increasing role recently 
\citep{Leifer2013,Chaves2015a,Ried15,Spekkens2015,AICarrowoftime}.

\paragraph{Acknowledgements}
Thanks to Robin Evans for correcting remarks on an earlier version.


\end{document}